\title{On Universal Cycles for new Classes of Combinatorial Structures.}
\author{\small Antonio Blanca \\
\small Georgia Institute of Technology \\
\small \tt{ablanca3@gatech.edu} \and
\small Anant P.~Godbole\\
\small East Tennessee State University\\
\small \tt{godbolea@etsu.edu} }
\documentclass [12 pt]{article}
\usepackage{amsfonts}
\usepackage{graphicx}
\usepackage{amsthm}
\usepackage{amsmath}
\usepackage{amssymb}
\begin{document}
\def\diam{{\rm diam}}
\def\ess{\rm ess}
\def\p{{\mathbb P}}
\def\ep{\varepsilon}
\def\P{{\rm Po}}
\def\cf{{\cal F}}
\def\cl{{\cal L}}
\def\e{{\mathbb E}}
\def\v{{\mathbb V}}
\def\l{\lambda}
\def\ll{{\ell_n}}
\def\a{{\alpha_n}}
\def\ph{\varphi(n)\sqrt n}
\def\dist{{\rm dist}}
\def\lr{\left(}
\def\rr{\right)}
\def\cd{\cdot}
\def\ts{\thinspace}
\def\lc{\left\{}
\def\rc{\right\}}
\def\qed{\vbox{\hrule\hbox{\vrule\kern3pt\vbox{\kern6pt}\kern3pt\vrule}\hrule}}
\newcommand{\hyp}{\mathcal{H}}
\newcommand{\hypp}{\mathcal{K}}
\newcommand{\Nu}{$\begin{Large}$\nu$\end{Large}$}
\newcommand{\ignore}[1]{}
\newtheorem{thm}{Theorem}
\newtheorem{result}{Result}
\newtheorem{lemma}[thm]{Lemma}
\newtheorem{cor}[thm]{Corollary}
\newtheorem{prop}[thm]{Proposition}
\maketitle

\begin{center}
\noindent{\bf Abstract}
\end{center}

\noindent A universal cycle (u-cycle) is a compact listing of a collection of combinatorial objects. In this paper, we use natural encodings of these objects to show the existence of u-cycles for collections of subsets, matroids, restricted multisets, chains of subsets, multichains, and lattice paths. For subsets, we show that a u-cycle exists for the $k$-subsets of an $n$-set if we let $k$ vary in a non zero length interval. We use this result to construct a ``covering" of length $(1+o(1))$$n \choose k$ for all subsets of $[n]$ of size exactly $k$ with a specific formula for the $o(1)$ term. We also show that u-cycles exist for all $n$-length words over some alphabet $\Sigma,$ which contain all characters from $R \subset \Sigma.$ Using this result we provide u-cycles for encodings of Sperner families of size 2 and proper chains of subsets.

\section {Introduction}

A universal cycle is a cyclic sequence which contains each element of a collection of combinatorial objects exactly once. A classic example is the de Bruijn cycle of order $n.$ If $n = 3,$ the binary cyclic string  11101000 contains the eight possible 3-length binary strings as a substring exactly once. Another universal cycle is 1234524135 which contains all 2-subsets of $[n] = \{1,2,3,4,5\}.$ 

\noindent \newline The authors of [1] conjectured that u-cycles exists for all $k$-subsets of $[n]$ provided that $k$ divides $n-1 \choose k-1$ and $n \ge n_0(k).$ Progress on this conjecture has been slow; a good compilation of partial results for this conjecture appears in [2] showing that most of the work is yet to be done. This motivated the work in by Curtis et al. [3] in which packings of length $(1 - o(1))$$n \choose k$ for $k$-subsets were found. (A u-cycle packing is a sequence in which each element appears at most once implying that some elements may be missing from the listing.)

\noindent \newline Given the difficulty for proving the conjecture, we adopt a different approach to the problem. So far we have been explicitly listing the $k$-subsets of $[n],$ but we could do otherwise. A binary string of length $n$ containing exactly $k$ ones provides a unique encoding for a particular $k$-subset of $n.$ A u-cycle for all binary strings of length $n$ containing exactly $k$ ones will do the trick by providing a compact listing for all $k$-subsets of $[n].$ Notice that the length of the new u-cycle would also be $n \choose k$ in most cases. It is worth pointing out that this encoding of subsets of $[n]$  is quite popular and widely used.

\noindent \newline Using this encoding we show that u-cycles exists for all $k$-subsets of $[n]$ if we let $k$ vary in some non zero length interval. This allows us to show the existence of a covering of length $(1+o(1))$$n \choose k$ for the $k$-subsets of $[n].$ The covering result complements nicely the work on packings in [3]. Unfortunately, not even using the advantages that this encoding provides were we able to show the existence of u-cycles for all $k$-subsets of $[n].$ The difficulty arises from the fact that once we fix the first $n$-length binary string containing $k$ ones, the $(n+1)$th character would be forced, and we end up with $\frac{{n \choose k}}{n}$ packings each of length $n.$ 

\noindent \newline Using similar encodings for collections of combinatorial structures we manage to find universal cycles for restricted multisets, $k$-chains of subsets of $[n]$, and multichains (set of chains) with a common ending subset. 

\noindent \newline While finding u-cycles for other structures, it was noticed that on many occasions we were just proving the existence of u-cycles for $n$-length words over some alphabet $\Sigma$ but with some restrictions. This leads to Theorem 11 in which we show that there exists a u-cycle for all $n$-length words over $\Sigma$ containing all characters from $R \subset \Sigma.$ Some applications of this result follow by showing the existence of universal cycles for natural encodings of Sperner families of size 2 and proper chains of subsets of size 2 as well.

\noindent \newline In the final section of this paper we show the existence of u-cycles for all 2-dimensional lattice paths of length $n$ starting at the origin and ending at a distance of at most $k$ from it. In this case, we deal with a more ``explicit" representation of the structure in question, and we believe that the strategies in use are extendable to other similar paths and to more dimensions.  

\section {Universal cycles of subsets}

Any subset of $[n] = \{1,...,n\}$ can be encoded as a binary string of length $n$. Let $2^{[n]}$ be the power set of $[n]$, and let $A \subseteq 2^{[n]}.$ A universal cycle of $A$ is a cyclic sequence of length $|A|$ such that  for all $a$ in $A$ its corresponding binary string appears exactly once in the sequence. For this representation of subsets, having $x$ ones in a binary string is equivalent to having a binary string of sum $x,$ and we will use both terminologies.

\noindent \newline There are some parcial results for this encoding of subsets of $[n]$ following from the work in [5]. If we consider only odd values for $n,$ then there exists a u-cycle for binary strings with the number of ones differing only by one from the number of zeros. These binary strings encode all subsets of size $\lceil\frac{n}{2}\rceil$ and $\lfloor\frac{n}{2}\rfloor.$ We follow to generalize this idea. 

\begin{thm}
Let $A$ be the set of all subsets of [n] with size between $p$ and $q,$ for $p < q$. Let $A'$ be the set of the corresponding strings for $A$. There exists a universal cycle for $A'$.
\end{thm}
\begin{proof} We start with a few simple remarks. Notice that $q \le n$ which implies that $p \le n - 1.$ If $p = n - 1$, then the binary string $s = 1...10$ of length $n+1$ will be a universal cycle. In this case, $|A'| = {n \choose n-1} + {n \choose n} = n+1 = |s|,$ and the binary string with $n$ ones as well as all the binary strings with $n-1$  ones are all contained in the cyclic string $s.$ So we may assume that $p < n-1.$

\noindent \newline We construct a digraph $D = <V,E>$ where $V$ is the set of all binary strings of length $n-1$. Let $v = a_1a_2...a_{n-1}$ and $u = b_1b_2...b_{n-1} \in V$; there is a directed edge $(v,u) \in E$ if and only if $a_i = b_{i-1}$ for $i = 2,...,n-1$ and $b_{n-1} + \sum\limits_{i=1}^{n-1} a_i = k$ for some $k$ such that $p \le k \le q.$  Thus, every edge in $E$ corresponds to a unique element of $A'$ (and viceversa) establishing a bijection.

\noindent \newline $D$ is known as the overlap graph for $V.$ Let $P = \{v_0,v_1,...,v_k\}$ be a path in $D$ with $v_i \in V.$ Two adjacent vertices in $P,$ $v_i$ and $v_{i+1}$ will determine a $n$-length string by overlaping the $n-2$ characters they have in common. Hence, path $P$ will determine a string $S$ of length $n - 1 + k$ which contains $k$ elements of $A'$ just by following $P$ and adding the corresponding characters. It follows then that an Eulerian circuit in $D$ will determine a universal cycle for $A'$. The proof reduces to showing that $D$ is an Eulerian digraph. 

\noindent \newline {\bf Fact 1:} [4] A digraph $D = <V,E>$ is Eulerian if and only if $d^+(v) = d^-(v)$ for all $v\in V$ and the underlying graph is connected except for isolated vertices.

\noindent \newline Let $v = a_1...a_{n-1} \in V,$ and let $\sum\limits_{i=1}^{n-1} a_i = w.$ If $w > q$ then any edges incident to $v$ would determine an edge corresponding to a binary string not in $A',$ so $d^+(v) = d^-(v) = 0.$ Similarly, if $w < p - 1,$ then $d^+(v) = d^-(v) = 0.$ If $p \le w < q,$ then any $x \in \{0,1\}$ can be appended to $v$ and the edge will correspond to an element of $|A'|$, so $d^+(v) = d^-(v) = 2.$

\noindent \newline Now we just take care of the border cases. If $w = p - 1,$ then any incident edge to $v$ should append  a ``1", either at the end or at the beginning to form an $n$-length string with $p$ ones, so $d^+(v) = d^-(v) = 1.$ In the case where $w = q,$ incident edges should append  a ``0" so that the resulting $n$-length string has no more than $q$ ones, hence $d^+(v) = d^-(v) = 1.$ It follows then that $d^+(v) = d^-(v)$ for each $v \in V.$ 

\noindent \newline Next, we analyze the connectivity of $D.$ We will show that every vertex in $V$ is connected to one with exactly $p$ ones. If $v \in V$ has more than $p$ ones, then we can go through edges that append zeros to the end of $v$ until we get to some vertex of sum $p.$ If $v$ has less than $p$ ones, then it is either isolated, or it has exactly $p-1$ ones. Similarly, we go through edges that append ones until we get to a vertex of sum $p.$ Appending ``1" to the end of $v$ may or may not lead us to some other vertex with more ones than $v$ depending on whether or not its first character is ``0". A ``0" is guaranteed to appear in some point since we are assuming that $p < n-1.$

\begin{lemma} Let $v = 0...01...1 \in V$ be a vertex with $p$ ones. Then $v$ is connected to every other vertex in $V$ that has exactly $p$ ones.
\end{lemma}

\begin{proof} Doing $``01" \rightarrow ``10"$ swaps only we can get from $v$ to any other binary string with $p$ ones. In $D$, this swap of adjacent elements will be the equivalent of going through a valid path from $v$ to some vertex $u$ that looks just like $v$ except that a ``01" is swaped to a ``10". If this kind of swap is feasible, we guarantee that $v$ is connected to every other vertex of sum $p$.

\noindent \newline Let $v = a_1,..,a_{n-1} \in V$ be a vertex of sum $p.$ A left shift of $v$ is the vertex $v_l = a_2,...,a_{n-1},a_1.$ If $v$ has sum $p$, then the edge $(v,v_l) \in E$ because $\sum\limits_{i=1}^{n-1} a_i + a_1 = p + a_1 \le q,$ and the overlap between $v$ and $v_l$ matches correctly. So any vertex of sum $p$ is connected to all other vertices whose binary string can be obtained from $v$ by some number of left shifts.

\noindent \newline If vertex $v$ contains ``01" at some position $i$ then $v = a_1...a_{i-1}01a_{i+2}...a_{n-1}.$ Since by left shifting $v$ we can get to $v' = 01a_{i+2}...a_{n-1}a_1...a_{i-1},$ then there exists a path from $v$ to $v'.$ Now, let $v'' = 1a_{i+2}...a_{n-1}a_1...$ $a_{i-1}1$ and $v''' = a_{i+2}...a_{n-1}a_1...a_{i-1}10$. Both edges, $e = (a',a'') $ and $e' =(a'',a''')$ are in $E$ because the vertices match accordingly and the corresponding $n$-length string for $e$ and $e'$ have $p+1$ and $p$ ones respectively. The vertex $v'''$ has $p$ ones, so by left shifting it we can get to the vertex $w = a_1...a_{i-1}10a_{i+2}...a_{n-1}$ and this concludes of proof for this lemma.\end{proof}

\noindent \newline So, as every vertex of $D$ is either isolated or connected to one of sum $p,$ and every vertex of sum $p$ is reachable from one particular vertex, $D$ has at most one nontrivial component. Given that for all $v \in V, d^+(v) = d^-(v),$ by Fact 1 $D$ is Eulerian, and hence there is a universal cycle for $A'$.
\end{proof}
\noindent \newline The fact that there exists a universal cycle for all the subsets whose size is in an interval allow us to identify universal cycles for other fundamental combinatorial objects. We will illustrate this by proving the existence of universal cycles for a specific kind of matroid. 

\noindent \newline Matroids were introduced by Whitney in 1935. A matroid $M$ consists of a pair $(E,\Gamma)$ where $E$ is a finite set and $\Gamma$ is a set of subsets of $E$ such that [7]:

\noindent i) $\Gamma$ is not empty.

\noindent ii) Every subset of each member of $\Gamma$ is also in $\Gamma.$

\noindent iii) If $X$ and $Y$ are in $\Gamma$ and $|X| = |Y| + 1,$ then $\exists x \in X-Y$ s.t. $Y \cup \{x\}$ is also in $\Gamma.$

\begin{cor}
Let $M = (E,\Gamma)$ be the matroid where $E = [n]$ and $\Gamma$ is the set of all subsets of $[n]$ of size less than or equal to $q.$ There exists a universal cycle for the set of binary encodings of the elements of $\Gamma.$
\end{cor}

\noindent{\bf Proof.} $\Gamma$ contains all the subsets of $[n]$ of size $k$ where $0 \le k \le q.$ Since a universal cycle exists for all subsets with size in a given range, it is guaranteed that a universal cycle for the binary codes of all the elements in $\Gamma$ exists.

\noindent \newline So far, we have found universal cycles for binary encodings of some subset $A$ of the power set of $[n].$ If instead we look at uncoded universal cycles for $A$ we find that some important work has been done.

\noindent \newline If $A$ is the set of all $k$-subsets of $[n],$ in the early 1990's the authors of [1] conjectured that universal cycles exists for $A$ for some $n$ sufficiently large. However, progress on this conjecture has not occured since the work of Hurlbert [9] and B. Jackson [10], which motivated the work by Curtis et al. [3].

\noindent \newline The authors of [3] realized the difficulty of proving the conjecture stated in [1], so they provided a different approach. They defined a universal packing for $k$-subsets of $[n]$ as a cyclic sequence which contains each $k$-subset at most once. For example the cyclic sequence ``12345" is a universal packing when $n = 5$ and $k = 3$ which contains 5 of the 10 possible 3-subsets for $[n]$. Curtis et al. found a construction for universal packings of $k$-subsets which guarantees a u-cycle of length $(1 - o(1))$$n \choose k$. We propose a different approach. Instead of looking at universal packings we define universal coverings.

\noindent \newline {\bf Definition} Given a set $A,$ a universal covering will be a cyclic sequence which contains each element of $A$ at least once. 

\noindent \newline Using our binary encoding, a universal covering for the $k$-subsets of $[n]$ will be a binary cyclic string which contains all $n \choose k$ $n$-length binary strings corresponding to the $k$-subsets of $[n]$ as substrings.

\begin{thm}
Let $A$ be the set of all $k$-subsets of [n] and let $A_k$ be the set of corresponding binary strings for $A.$ There exists a universal covering for $A_k$ of length (1+o(1))$n \choose k$.
\end{thm}

\begin{proof} Let $A_{k-1}$ be the set of all the binary encodings for the $(k-1)$-subsets of $[n],$ and let $A_{k-1,k} = A_{k-1} \cup A_{k}$. Then $A_{k-1,k}$ will be the set of all encodings for subsets of size $k-1$ and $k,$ and by Theorem 1 there exists a universal cycle $U$ for $A_{k-1,k}$. $U$ is a universal cover for $A_k.$ Also,

\noindent \newline $\lim\limits_{n \rightarrow \infty} \frac{|U|}{|A_k|}$ = $\lim\limits_{n \rightarrow \infty} \frac{{n \choose k} + {n \choose k-1}}{{n \choose k}}$ = $\lim\limits_{n \rightarrow \infty} (1 + \frac{k}{n - k + 1}) = 1$ 

\noindent \newline From this calculation it follows that $|U| = (1 + o(1))$$n \choose k$, where $o(1) = \frac{k}{n - k + 1}$.\end{proof}

\section {Universal cycles of multisets} A universal cycle of a set $A$ of multisets of $[n]$ is a cyclic sequence of length $|A|$ such that  for all $a \in A$ its corresponding string appears exactly once in the sequence. Hurlbert et al. conjectured in [11] that a u-cycle exists for all $k$-multisets of $[n]$ given that $n$ is large enough and divides ${n + k - 1} \choose {k}$. They were able to prove that the conjecture holds for $k = 2$ or $3.$ We adopt a similar strategy for multisets as we did previously for subsets; given the difficulty of finding u-cycles for $A,$ we analize the existance of u-cycles for natural encodings of $A$.

\noindent \newline A multiset $M$ with elements taken from $[n]$ can be encoded by the string $a_1a_2...a_n$ where $a_i \ge 0$ represents the number of times element $i$ appears in $M$. A multiset $M$ is $t$-restricted if every element of $[n]$ appears no more than $t$ times in $M$.

\begin{thm}
Let $A$ be the set of all $t$-restricted $k$-multisets of [n] such that $n > 1,$ $p \le k \le q$, $t = q - p$, and $2p \le q$. Let $A'$ be the set of corresponding strings for $A$. There exists a universal cycle for $A'$.
\end{thm}
\medskip

\begin{proof}  If $p = q = 0,$ then a u-cycle trivially exists. Since $p$ and $q$ are both positive it follows that $p < q.$ Let $V$ be the set of all strings of length $n-1$ over the alphabet $\{0,1,...,t\}$, and let $D=<V,E>$ be a digraph. Let also $v = a_1a_2...a_{n-1}$ and $u = b_1b_2...b_{n-1}$ be elements of $V.$ There is a directed edge $(v,u) \in E$ if and only if $a_i = b_{i-1}$ for $i = 2,...,n-1$ and $\sum\limits_{i=1}^{n-1} a_i + b_{n-1} = k$ for some $k$ such that $p \le k \le q.$ An Eulerian circuit in $D$ will determine a universal cycle for $A'$, so once again the proof reduces to show that $D$ is Eulerian.

\noindent \newline Let $v = a_1a_2...a_{n-1} \in V$ such that $\sum\limits_{i=1}^{n-1} a_i = w$. If $w > q$ already, then $d^+(v) = d^-(v) = 0$ because nothing can be added to $v$ to get a string in $A'$. If $w = q$, then it is only possible to add a 0 to $v$ since adding anything else will make the $n$-length string have sum greater than $q$, so $d^+(v) = d^-(v) = 1.$ If $p \le w < q$, then we will have edges coming in and out of $v$ for every $x \ge 0$ such that $w+x \le q$, hence $d^+(v) = d^-(v) = q-w+1.$ If $w < p$, then the valid edges will be those corresponding the values of $x$ such that $p \le w + x \le q$ hence $d^+(v) = d^-(v) = q-2p+w+1$ because $p \le q - p.$ Thus, $d^+(v) = d^-(v)$ for all $v \in V$.

\noindent \newline We will prove the weak connectivity of $D$ in two steps.

\begin{lemma} Every $v \in V$  is either isolated or connected to some vertex $u = b_1...b_{n-1} \in V$ such that $\sum\limits_{i=1}^{n-1} b_i = p.$
\end{lemma}

\begin{proof}  Let $v = a_1...a_{n-1}$ and $\sum\limits_{i=1}^{n-1} a_i = w.$ If $w = p$ then the result is trivial because every vertex is vacuosly connected to itself, and if $w > q$ then $v$ is isolated. If $p < w \le q$, let $i$ be the smallest value for which $a_1 + a_2 + ... + a_i \ge w - p.$ Such $i$ exists because $w > p.$ Since $a_1 + ... + a_{i-1} < w - p$ we can follow the path $P = a_1...a_{n-1} \rightarrow a_2...a_{n-1}0 \rightarrow ... \rightarrow a_i...a_{n-1}0...0 \rightarrow a_{i+1}...a_{n-1}0...0b$ where $b = p - w + (a_1 + ... + a_i).$ This is possible because every pair of consecutive vertices in $P$ determine an $n$-length string with sum greater than $p$ and less than $q$. The last vertex in $P$ has sum exactly $p$. Hence if $p < w \le q,$ then $v$ is connected to a vertex of sum $p$. Notice that $b \le q - p.$

\noindent \newline Now, if $w < p$, we will go from $v$ through vertices of equal or higher sum until we reach a vertex of sum $p$. The idea would be to go through the edges adding $t = q - p$ to the sequence until possible. So, let $i$ be the smallest value for which $(q - p - a_1) + ... + (q - p - a_i) \ge p - w.$ We need to show first that such $i$ is guaranteed to exist. Suppose it doesn't, then:
\begin{eqnarray}
(q - p - a_1) + ... + (q - p - a_{n-1}) < p - w\nonumber\\
(q - p)(n-1) - (a_1+...+a_{n-1}) < p - w\nonumber\\
q(n-1) <  np\nonumber\\\nonumber
\end{eqnarray}
which does not hold given that $2p \le q$ and $n \ge 2.$

\noindent \newline Note that the path $P = a_1....a_{n-1} \rightarrow a_2...a_{n-1}t \rightarrow ... \rightarrow a_i...a_{n-1}t...t \rightarrow a_{i+1}...a_{n-1}$ $t...tb$ where $b = p - w - i(q - p) + (a_1 + ... + a_i)$ is feasible in $D$ because we are always adding $t =q - p,$ which given that $p \le q - p,$ guarantees that $w + (q - p) \ge p.$ Also, each vertex in $P$ has sum at most $p-1$ (except the last one); since $(p - 1) + (q - p) < q,$ then every edge in $P$ is in the valid interval. Given that $P$ exists,we see that $v$ is also connected to some vertex of sum $p$ if $w < p,$ and the proof of this lemma is complete. \end{proof}

\begin{lemma} Every vertex of sum $p$ in $D$ is connected to $v = 0...0p \in V.$
\end{lemma}

\begin{proof} Notice first that  $v$ is a valid vertex given that $p \le q - p.$ Let $u = a_1a_2...a_{n-1} \in V$ be some vertex of sum $p,$ and let $i$ be the smallest value for which $a_i \neq 0.$ So, $u = 00...a_i...a_{n-1}$ and $\sum\limits_{j = i}^{n-1} a_j= p.$ We will prove that $u$ is connected to $u^* = 0...0(a_i + a_{i+1})...a_{n-1}.$ From this it follows that by repeating this process $n - 1 - i$ times we get from $u$ to $v$ since $u^*$ will also be a vertex of sum $p.$

\noindent \newline Left shifts (as defined before) are possible in any vertex of sum $p.$ So, by doing left shifts on $u$ we prove that it is connected to $u' = a_ia_{i+1}...a_{n-1}0...0.$ 

\noindent \newline Let $u'' = a_{i+1}a_{i+2}...a_{n-1}0...0 \in V.$ Notice then that $(u',u'') \in E$ because $u'$ and $u''$ match properly and $\sum\limits_{j=i}^{n-1} a_j = p.$ The vertex $u''$ has sum $p - a_i$. Let $u''' = a_{i+2}...a_{n-1}0...0(a_i + a_{i+1}),$ and notice that $(u'',u''') \in E$ because once again both vertices match properly and $\sum\limits_{j = i+1}^{n-1} a_j + (a_i + a_{i+1}) = p + a_{i+1} \le q$ given that $a_{i+1} \le p \le q - p.$  Finally, by doing left shifts on $u'''$ we can get to $u^* = 0...0(a_i + a_{i+1})...a_{n-1},$ and this concludes the proof for the lemma. \end{proof}

\noindent \newline So, given that every vertex of $D$ is either isolated or connected to some vertex of sum $p$, and that all vertices of $D$ of sum $p$ are connected to one common vertex, it follows that the underlying graph for $D$ has only one nontrivial component. Thus D is Eulerian, and this guarantees the existence of a universal cycle for $A'.$ \end{proof}

\section {Universal cycles of chains of subsets}

We next analyze another interesting set of combinatorial objects. Given $[n] = \{1,2,...,n\},$ a $k$-chain of $[n]$ is a sequence of sets $A_1 \subseteq A_2 \subseteq ... \subseteq A_k \subseteq [n].$ The number of distinct $k$-chains of $[n]$ is $(k+1)^n.$ We want to create a natural encoding for the $k$-chains of $[n],$ and analyze then the existence of universal cycles.

\noindent \newline Each $A_i$ of some $k$-chain can be encoded by a binary string of length $n$. In this case, for convenience we will think of that binary string as an ordered binary $n$-tuple denoted by $s_i$. We will define the $\oplus$ operation among these $n$-tuples as follows: $s_i \oplus s_j = s'$ if and only if $s'[k] = s_i[k] + s_j[k]$ for all $k$ such that $1 \le k \le n.$ For instance, let $a = \{1,2,3\},$ $b = \{2,2,2\},$ and $c = \{1,1,1\}$ then $a \oplus b \oplus c = \{4,5,6\}.$

\begin{lemma}  Let $C = A_1 \subseteq A_2 \subseteq ... \subseteq A_k$ be a k-chain. The $n$-tuple $S_c = s_1 \oplus s_2 \oplus ... \oplus s_k$ uniquely identifies $C.$
\end{lemma}

\begin{proof}  The proof is quite straightforward. For every $k$-chain there is a unique $n$-tuple $S_c$. Suppose $S_c[i] = t,$ since $A_i \subseteq A_{i+1}$ then the only way in which $S_c[i] = t$ is if $s_j[k] = 1$ $\forall j$ such that $k - t + 1 \le j \le k,$ so for every $n$-tuple there is also a unique $k$-chain. \end{proof}

\begin{thm}
Let $A$ be the set of all $k$-chains of $[n]$, and let $A'$ be the set of the $n$-tuples for $A$. There exists a universal cycle for $A'$.
\end{thm}

\begin{proof} We mentioned before that $|A| = |A'| = (k+1)^n.$ Let $a \in A',$ and notice that $0 \le a[i] \le k$ for all valid $i$. So, any element of $A'$ uniquely corresponds to a word of length $n$ over the alphabet $\Sigma = \{0,...,k\}.$ Let $\Sigma_n$ be set of all words of length $n$ over $\Sigma,$ so $|\Sigma_n| = (k+1)^n.$ It follows then that $A' = \Sigma_n$, so a universal cycle for $\Sigma_n$ will determine a universal cycle for $A'$. Since universal cycles are known to exists for words of any length over any alphabet, there exists a universal cycle for $A'.$ \end{proof}

\noindent \newline A natural follow up would be to analyze sets of chains of $[n],$ so let's call these kind of sets $multichains.$ We are particularly interested in multichains such that every chain in the multichain shares the same last subset $D$. So, let $S = \{C_1,C_2,...,C_p\}$ be a multichain where $|C_i| = l_i,$ and $C_i = C_{i,1} \subseteq ... \subseteq C_{i,l_i-1} \subseteq D$ for all $i = 1,...,p.$  Notice that $D$ is not fixed initially, but once we choose some $D$ all $C_i \in S$ will have $D$ at the end. As we did for chains, we first need to encode the multichains, and we will borrow some ideas from what we have done before.

\noindent \newline $S_i$ will denote the corresponding $n$-tuple for each $C_i$ using the $\oplus$ operation. Let $D$ be any subset of $[n]$ and $s_d$  be its corresponding binary string. If $s_d[k] = 0,$ then all chains ending at $D$ will have zero at position $k.$ If $s_d[k] = t,$ all chains ending at $D$ will have the $t$ last sets with a one at position $k$ where $t$ varies between 1 and the length of the chain. This provides the intuition for the encoding of these multichains. For each multichain $S$ we will create an ordered $n$-tuple $N$ where each element of $N$ will be in turn an ordered $p$-tuple. If $s_d[k] = 0,$ then $N[k] = (0,...,0)$, otherwise $N[k] = (t_1,...,t_p)$ where $t_i = S_i[k].$

\noindent \newline For example, let $[4] = \{1,2,3,4\},$ $A_1 = \{1,4\},$ $A_2 = \{2,4\},$ and $A = \{1,2,4\}.$ Let's define also two simple chains $C_1 = A_1 \subseteq A,$ $C_2 = A_2 \subseteq A,$ and the multichain $S = \{C_1,C_2\},$ then the encoding for $S$ would be $N = \{\{2,1\},\{1,2\},\{0,0\},\{2,2\}\}.$ It is not hard to convince ourselves that the described encoding uniquely identifies a multichain.

\begin{thm}
 Let $A$ be then the set of all possible multichains $S$ over $[n],$ and let $A'$ be the set of the encodings for $A$. There exists a universal cycle for $A'$.
\end{thm}

\begin{proof} The proof idea is similar to that for a single chain. Let $S = \{C_1,C_2,...,$ $C_p\}$ , $L = \{l_1,...,l_p\}$ where $l_i = |C_i|,$ and let $N$ be the $n$-tuple encoding $S$. Then, $N[k] = (t_1,...,t_p),$ and notice that $1 \le t_i \le l_i$ for any $k$ unless $N[k] = (0,...,0)$. So, the number of distinct $p$-tuples is $\prod\limits_{i=1}^{p} l_i + 1.$ Let $P$ be the set of all such $p$-tuples, and let $\Sigma'$ be some alphabet which uniquely assign a symbol to each element of $P$. Then, any word of length $n$ over $\Sigma'$ will uniquely determine an encoding $N$ of some multichain $S$. So, since it is known that there is a universal cycle for all words of any length over any alphabet, then there is a universal cycle for $A'.$ \end{proof}

\section {Universal cycles of words over restricted alphabets}

Given some alphabet $\Sigma$ it is known that all words of any length can be u-cycled. Let $A_n$ be the set of all words of a fixed length $n$ over $\Sigma.$ We are interested in imposing some restrictions on $\Sigma$ i.e. analyzing the existence of u-cycles for some $A'_n \subset A_n.$ 

\noindent \newline Some u-cycles were found in [8] for different kinds of restrictions on $\Sigma$. For example, a $password$ is an $n$-length word over $\Sigma$ which contains at least one character from $q$ distinct subsets (classes) of $\Sigma.$ It was shown in [8] that u-cycles for passwords exist as long as $2q \le n.$ We notice that in the case where each class is of size 1, such strong restriction is not necessary.

\begin{thm}
Let $X \subseteq \Sigma$ and $R \subseteq \Sigma$ such that $X \cap R = \emptyset$ and $|R| < n.$ Let $A'_n$ be the set of all words of length $n$ over $\Sigma$ such that $w \in A'_n,$ if and only if $w$ contains all characters from $R$ and none of the characters from $X.$ There exists a u-cycle for $A'_n.$ 
\end{thm}

\begin{proof} We will call $X$ the exclusion set and $R$ the required set. We could see $R$ as the set containing $|R|$ distinct classes of size one. Notice that all words in $A'_n$ are actually over the alphabet $\Sigma' = \Sigma - X,$ so we are interested in u-cycling all words of length $n$ over $\Sigma'$ which contain all characters from $R.$ 

\noindent \newline Let us create the overlap graph $D = <V,E>$ for $A'_n$ where $V$ is the set of all $(n-1)$-length strings over $\Sigma'.$ Let $v = a_1a_2...a_{n-1}$ and $u = b_1b_2...b_{n-1}$; there is a directed edge $(v,u) \in E$ if and only if $v,u \in V,$ $a_i = b_{i-1}$ for $i = 2,...,n-1,$ and $R \subseteq \{a_1,...,a_{n-1},b_{n-1}\}.$ Once again proving that $D$ is Eulerian is sufficient  to guarantee the existence of a u-cycle for $A'_n.$
  
\noindent \newline Let $v \in V.$ If $v$ is missing more than one character from $R$ then $d^+(v) = d^-(v) = 0.$ If it is missing only one character then it can only have incident edges which append the missing character, so $d^+(v) = d^-(v) = 1.$ If $v$ already has all characters from $R$ then anything can be appended, so  $d^+(v) = d^-(v) = |\Sigma'|.$ Then, $\forall v \in V$  $d^+(v) = d^-(v) .$

\noindent \newline We next analyze the connectivity of $D.$ Let $R = \{r_1,...,r_k\},$ and let $\rho = r_1r_2...r_k$ be a $k$-length string. We call $\delta = \rho\rho...\rho q \in V$ the {\it vertex of maximal distributed density} where $q = r_1r_2...r_b$ for some $1 \le b \le k$. Given that $k \le n-1,$ we have $n-1 = ak + b,$ so the number of times $\rho$ appears in $\delta$ is $a,$ $|q| = b,$ and $|\delta| = n-1.$

\begin{lemma} If $v \in V,$ then $v$ is either isolated or it is connected to $\delta.$
\end{lemma}

\begin{proof} If $v$ is missing more than one character from $R$ then $v$ is isolated, so we may assume that $v$ is missing at most one character from $R.$ If $v$ is not missing any character from $R$ then we denote by $\rho'$ to the $k$-string which contains all $k$ characters from $R$ in the order they first appear in $v.$ If it is the case that $v$ is missing some character $a \in R,$ then $\rho'$ will start with $a$ and will then contain all other $k-1$ characters from $R$ in the order they first appear in $v.$ In any case, $\rho'$ is a permutation of $\rho.$

\noindent \newline Notice now that starting at $v$ we can go through the edges that append $\rho'$ to the end of $v,$ because by adding $\rho'$ we guarantee that each element from $R$ appears at least once in each of those edges. So, after appending $\rho',$ we get to some vertex $v' = a_1...a_{n-1-k}\rho' \in V.$ We will then repeatedly append $\rho$ to $v'$ to reach $v'' = \rho'\rho...\rho q.$ Remember that $n-1 = tk + r,$ so $\rho$ will be added exactly $t-1$ times. Let's analyze why it is possible to get to $v''$ from $v'$. In this case, our claim could be even stronger; from $v' = a_1...a_{n-1-k}\rho'$ we can get to any $u = \rho'b_1...b_{n-1-k}$ because the direct path  $v \leadsto u$ will generate the string $S = a_1...a_{n-1-k}\rho'b_1...b_{n-k-1},$ and any $n$-length substring of $S$ fully contains  $\rho'.$ So $v$ is connected to $v'',$ and we next show that $v''$ is connected to $\delta.$

\noindent \newline Notice that $v''$ contains every element from $R$ at least once. As noted, the difference between $v''$ and $\delta$ is that $\rho'$ is some permutation of $\rho.$ By doing only swaps of adjacent elements one can get from one permutation to any other, so by proving that $v''$ is adjacent to $u = \rho''\rho...\rho q$ where $\rho''$ differs from $\rho'$ by an adjacent swap we guarantee that $v''$ is connected to $\delta.$

\noindent \newline So, let $\rho' = a_1...a_ia_{i+1}...a_{k}$, and let $\Lambda = \rho...\rho q;$ then $v'' = \rho'\Lambda.$ Left shifts are feasible in $v''$ since $\rho'$ contains all elements from $R.$ Thus $v''$ is connected to $v_1 = a_ia_{i+1}...a_{k}\Lambda a_1...a_{i-1}.$ Now, let $v_2 = a_{i+1}...a_{n-1}\Lambda a_1...$ $a_{i-1}a_{i+1}$ and $v_3 = a_{i+2}...a_{n-1}\Lambda a_1...a_{i-1}a_{i+1}a_i.$ The edge $(v_1,v_2) \in E$ because both vertices match properly and the corresponding $n$-length string contains every character from $R.$ Similarly, $(v_2,v_3) \in E,$ so $v''$ is connected to $v_3$ which by left shifts is connected to $u,$ hence $v''$ is connected to $u.$ By doing as many swaps as necessary $v''$ is connected to $\delta$ which concludes the proof for the lemma. \end{proof}

\noindent \newline Given that every vertex in $V$ is either isolated or connected to $\delta,$ then $D$ has a single nontrivial component. Since $d^+(v) = d^-(v)$ for all $v \in V,$ $D$ is Eulerian, and there exists a universal cycle for $A'_n.$ \end{proof}

\noindent \newline Next we illustrate how can we use the restricted alphabet theorem. Given $[n] = \{1,...,n\},$ the Sperner families of size 2 will consist of all $(A_1,A_2) \in 2^{[n]} \times 2^{[n]}$ such that $A_1 \not\subset A_2$ and $A_2 \not\subset A_1.$ $A_1$ and $A_2$ can be represented by a binary string of length $n.$

\noindent \newline Now let's define the $\odot$ operation between binary strings where $s = s_1 \odot s_2$ if and only if
\begin{eqnarray}
s[i] = \begin{cases} 0 & \mbox{if } s_1[i]s_2[i] = ``00" \\ 1 & \mbox{if } s_1[i]s_2[i] = ``01" \\ 2 & \mbox{if } s_1[i]s_2[i] = ``10"\\ 3 & \mbox{if } s_1[i]s_2[i] = ``11" \end{cases}\nonumber\\
\end{eqnarray}

\noindent \newline Given a Sperner family of size two $S = (A_1,A_2) \in [n] \times [n],$ $s = s_1 \odot s_2$ uniquely encodes S. 

\begin{cor}
Let $A'$ be the set of all the encodings for all Sperner families of size two. There exists a u-cycle for $A'.$
\end{cor}

\begin{proof} Let $A_1$ and $A_2$ be any two subsets of $[n]$ with corresponding binary strings $s_1$ and $s_2.$ Let $s = s_1 \odot s_2.$ It follows then $s$ will be a string  over $\Sigma = \{0,1,2,3\}.$ $A_1 \not\subset A_2$ implies that there exists $x$ such that $x \in A_1$ and $x \not\in A_2,$ so $s$ will contain a ``2''. By a similar analysis we can conclude that $s$ will also contain a ``1''. So, we are interested in finding a u-cycle for all strings over $\Sigma$ which contains all elements from $R = \{1,2\}$ which is guaranteed to exist by the restricted alphabet theorem. \end{proof}

\noindent \newline Another use of the restricted alphabet theorem comes from finding universal cycles for 2-chains of proper subsets. Let $A$ be the set of all pairs $S = (A_1,A_2) \in 2^{[n]}  \times 2^{[n]}$ such that $A_1 \subset A_2.$ Let $s_1$ and $s_2$ be the corresponding binary string for $A_1$ and $A_2$ respectively. The string $s = s_1 \odot s_2$ provides a unique encoding for $S.$

\begin{cor}
Let $A'$ be the set of all the encodings for all 2-chains of proper subsets of $n$. There exists a u-cycle for $A'.$
\end{cor}

\begin{proof}  Let $A_1$ and $A_2$ be any two subsets of $[n]$ with corresponding binary strings $s_1$ and $s_2.$ Let $s = s_1 \odot s_2.$ It follows then $s$ will be a string  over $\Sigma = \{0,1,2,3\}.$ $A_1 \subset A_2$ $\implies$ $\exists x$ s.t. $x \in A_2$ and $x \not\in A_1,$ so $s$ will contain a ``1''. Also, $\forall x \in A_1 \implies x \in A_2,$ so $s$ will have no ``2"s. So, we interested in finding a u-cycle for all strings over $\Sigma$ which contains all elements from $R = \{1\}$ and no element from $E = \{2\},$ which is guaranteed to exist by the restricted alphabet theorem. \end{proof}

\noindent \newline The restricted alphabet theorem can be used in a fashion similar to Collorary 13 and 14  to form u-cycles for a variety of other well studied combinatorial set systems such as intersecting set systems of two or more sets. We believe that this is a tool which could exploited to find u-cycles for other combinatorial structures.

\section {Universal Cycles of Lattice Paths}

Lattice paths are a heavily studied set of combinatorial objects (e.g. [12]). We continue our investigation by exploring the existence of universal cycles for some kind of lattice paths. In the previous sections we expanded on the study of some combinatorial structures for which the existance of u-cycles have been analyzed before. In this section, we try to motivate the analysis of u-cycles for lattice paths by analizing a specific kind. 

\noindent \newline In general a lattice path of length $n$ is a sequence of points $P_1,...,P_n$ of $\mathbb{Z} \times \mathbb{Z}$ such that the Manhattan distance between $P_i$ and $P_{i+1}$ is 1 for $i = 1,...,n-1.$ It is common to use a string over the alphabet $\Sigma = \{N,S,E,W\}$ to describe these paths (see Figure 1).
\begin{center}
\includegraphics[]{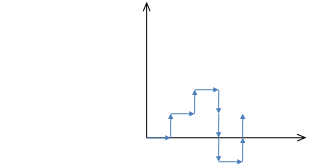}
\newline
Figure 1. Path corresponding to ``ENENESSSENN".
\end{center}

\noindent \newline Now suppose that we are given two positive integers $n$ and $k,$ and let $P_{n,k}$ be the set of the strings over $\Sigma$ corresponding to all lattice paths of length $n$ that start at the origin and end up at a distance of at most $k$ from it. In the case of $n = k = 3,$ $P_{3,3}$  will be the set of all paths ending at $(x,y)$ such that $|x| + |y| \le 3.$ Notice that there could be several ways of getting to the same $(x,y).$ For example we can get to $(1,2)$ through the paths: ``NNE'', ``NEN'', and ``ENN" all of length 3.

\noindent \newline The number of lattice paths of length $l$ starting at the origin and ending at $(x,y)$ is given by $\Lambda = {2n + x + y \choose n}{2n + x+ y \choose n + x}$ where $l = 2n + x + y$ [6]. From here it follows that $\Lambda \neq 0$ if $l \equiv x+y$ $mod$ $(2).$

\begin{thm} There exists a universal cycle for $P_{n,k}.$
\end{thm}

\begin{proof} Let $V$ be the set of all strings over $\Sigma$ for lattice paths of length $n-1$ which start at the origin and end at a distance $d \le k+1$ from it. Let's define the digraph $D=<V,E>$, and let $v = a_1...a_{n-1} \in V$ and $u = b_1...b_{n-1} \in V.$ The edge $(v,u) \in E$ if and only if $a_i = b_{i-1}$ for $i = 2,...,n-1,$ and $a_1...a_{n-1}b_{n-1} \in P_{n,k}.$ Proving the existence of universal cycles for $P_{n,k}$ is equivalent to showing that $D$ is Eulerian.

\begin{center}
\includegraphics[]{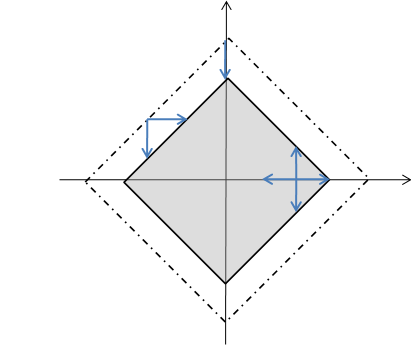}
\newline
Figure 2. Showing the $P_{n,k}$ region and vertices of different degrees.
\end{center}

\noindent \newline Let $v \in V,$ then $v$ is either in $P_{n,k}$ or one move away from being in it. The path corresponding to $v$ ends at some coordinates $(v_x,v_y)$ of the plane. From now on when we refer to the position or coordinates of a vertex in $V$, we mean the coordinates of its endpoint. Based on $(v_x,v_y),$ one can add 1, 2, or 4 distinct characters to $v$  to reach another vertex of $V$  through a valid edge (see Figure 2). But anything that can be appended at the end of $v$ to get a string in $P_{n,k}$ can also be appended at the beginning; hence $d^+(v) = d^-(v)$ for all $v \in V.$ We next analyze the connectivity of $D.$

\begin{lemma} Every vertex in $V$ has degree 4 or it is connected to some vertex of degree 4. 
\end{lemma}

\begin{proof} Let $v \in V.$ If $v$ has degree 4 we are done, so let $v$ have degree 1 or 2. We will prove first that every vertex of degree 1 is connected to some vertex of degree 2 or 4, reducing the proof for this lemma to show that every vertex of degree 2 is connected to some vertex of degree 4.

\noindent \newline Every vertex of degree 1 corresponds to a path which ends at a corner. Given that $k \ge 2,$ corners are at distance at least 4, meaning that we need at least two moves to get from one corner to another (In general the endpoints of adjacent vertices are at distance 0 or 2). If $v$ has degree 1 then it has only one neighbor $u \in V$ which could have degree 1, 2, or 4. If $u$ has degree 2 or 4 we are done with this first part of the proof, but if $u$ has degree 1, then $u$ corresponds to a path that ends in the same corner as $v$ given that distinct corners are at least two moves away from each other. From $u,$ we take the only possible move one more time. By repeating this move enough times, one is guarantee to leave this corner at some point ending at a vertex of degree 2 or 4.

\noindent \newline Let us clarify this idea with an example. Suppose vertex $v' \in V$ corresponds to a path ending at the ``north corner". From  the north corner we can only move south; moving south from $v'$ we can get to some other vertex in the north corner if and only if the corresponding string for $v'$ starts with an ``S" move. So, the only way of not getting out from the north corner is if $v'  = S....S,$ but to get to the north corner in the first place we should have moved at least $k$ times to the north. So, by moving south enough times we would leave the corner once one of those ``N" moves gets replaced by an ``S" move. Then every vertex of degree one is connected to some vertex of degree 2 or 4, and  it is sufficient to prove that every vertex of degree 2 is connected to a vertex of degree 4. 

\noindent \newline Let $d^+(v) = d^-(v) = 2,$ and let $(v_x,v_y)$ be the ending position of the path corresponding to $v.$ Without lost of generality we may assume $v_x \ge 0$ and $v_y \ge 0.$ Then the possible moves from $v$ are ``S" and ``W". Let $v = a_1...a_{n-1};$ we can left shift $v$ until we get to $v_1 = a_i...a_{n-1}a_1...a_{i-1}$ where $a_i$ is the first move of $v$ such that $a_i \neq$ ``S" and $a_i \neq$ ``W". Such $a_i$ is guaranteed to exist because $v$ was assumed to be in the first quadrant. If $a_i = $ ``N"  let $v_2 = a_{i+1}...a_{n-1}a_1...a_{i-1}W$ and if  $a_i = $``E" let $v_2 = a_{i+1}...a_{n-1}a_1...a_{i-1}S.$ In both cases, the edge $(v_1,v_2) \in E$ because the strings match appropriately, and the resulting $n$-length path is in $P_{n,k}.$  Notice that this is the case because $v_1$ is still in the first quadrant since it is just a reordering of $v,$ and moving ``S" or ``W" from a vertex in the first quadrant will not take us out of $P_{n,k}$. Hence, $v$ is connected to $v_2.$ But the coordinates of $v_2$ are $(v_x-1,v_y-1),$ so $d^+(v_2) = d^-(v_2) = 4.$ This concludes the proof for this lemma. \end{proof}

\begin{lemma} If $n$ is odd, every vertex in $V$ is connected to a vertex whose corresponding path ends at $(0,0),$ and if $n$ is even, every vertex in $V$ is connected to a vertex corresponding to a path ending at $(0,1)$.
\end{lemma}

\begin{proof} The parity of $n$ determines the existence of a path of length $n-1$ ending at the origin. Basically if $n$ is odd, there will be an $(n-1)$-length path starting and ending at the origin, but if $n$ is even we can only get to some vertex at distance 1 from the origin. 

\noindent \newline Since every vertex in $V$ has degree 4, or it is connected to some vertex of degree 4, it is sufficient to prove that any vertex of degree 4 is connected to some vertex at (0,0) or (0,1). So, choose $v \in V$ such that $d^+(v) = d^-(v) = 4.$ A vertex of degree 4 gives us plenty of choices. First, left shifts are always possible, and more importantly, any kind of adjacent swap is feasible. 

\noindent \newline Let $n_i,$ $s_i,$ $e_i,$ and $w_i$ be the number of ``N", ``S",``E", and ``W" moves in $v_i \in V$ respectively. Let $v_1 = N...NS...SE...EW...W \in V$ be a vertex with the same number of moves in each direction as $v.$ It is clear that by doing adjacent swaps one can get from $v$ to $v_1.$  Without losing generality we may assume $n_1 \ge s_1$ and $e_1 \ge w_1.$  Now, we will append ``S" moves to the end of $v_1$ until we get to some $v_2 = N...NS...SE...EW...WS...S \in V$ such that $0 \le n_2 - s_2 \le 1.$ Notice that $d^+(v_2) = d^-(v_2) = 4,$ and $v_2$ is at distance 1 or 0 from the $x$ axis. We can now reorder $v_2$ by doing adjacent swaps to get to $v_3 = E...EW...WN...NS...S.$ Since the number of ``E"s and ``W"s have not changed, we see that $e_1 = e_3 \ge w_3 = w_1,$ and we append ``W" at the end of $v_3$ until we get to $v_4 = E...EW...WN...NS...SW...W \in V$ where $0 \le e_4 - w_4 \le 1.$ 

\noindent \newline Let $(x_4,y_4)$ be the ending coordinates of the corresponding path for $v_4.$ If $n-1$ is even then $(x_4,y_4) = (0,0)$ or $(x_4,y_4) = (1,1),$ since we assume that $n_1 \ge s_1$ and $e_1 \ge w_1.$ From (1,1), we replace an ``E" move by a ``S" move and we get to some vertex in (0,0). So, if $n-1$ is even, every vertex in $V$ is connected to a vertex with coordinates $(0,0)$.

\noindent \newline Similarly,  if $n-1$ is odd then $(x_4,y_4) = (0,1)$ or $(x_4,y_4) = (1,0),$ since we assume that $n_1 \ge s_1$ and $e_1 \ge w_1.$ From (1,0), we replace an ``E" move by a ``N" move and we get to some vertex in (0,1). So, if $n-1$ is odd, every vertex in $V$ is connected to a vertex with coordinates (0,1), and this concludes the proof for this lemma. \end{proof}

\noindent \newline We just need to prove then that if $n-1$ is even, all vertices ending at (0,0) are connected, and if $n-1$ is odd, that all vertices ending at (1,0) are connected. Let for $n-1$ even $v_1,v_2 \in V$ be any two vertices corresponding to paths ending at the origin. Notice that $n_1 = s_1,$ $e_1 = w_1,$ $n_2 = s_2,$ and $e_2 = w_2.$ If $n_1 = n_2,$ then just by doing swaps one can get from $v_1$ to $v_2.$ Without losing generality assume that $n_1 > n_2.$ We reorder $v_1$ by doing adjacent swaps until we get to $v_3 = NSNS...NSEW...EW \in V$ where $n_3 = n_1.$ From $v_3$ it is possible to substitute the first ``NS" by a ``WE" by appending ``WE"s to the end. We can keep doing this until we get to $v_4 \in V$ such that $n_4 = n_2.$ Then by doing adjacent swaps $v_4$ is connected to $v_2,$ and thus so is $v_1.$ A similar analysis follows for two vertices at (0,1) when $n-1$ is odd. This complements the proof. \end{proof}

\section{Acknowledgements}
This work was done by Antonio Blanca under the supervision of Anant Godbole at the East Tennessee State University Research Experience for Undergraduates (REU) site during the summer of 2010. Support from NSF Grant 1004624 is gratefully acknowledged by the authors, as is the support offered by fellow REU participants. Also thanks to Andres Sanchez for a thoughtful review of the paper.


\begin{thebibliography}{99}
\bibitem {chung} F.R.K~Chung, P.~Diaconis, R.L.~Graham, ``Universal Cycles for combinatorial structures," {\it Discrete Math.} {\bf 110} (1992), 43--59.
\bibitem {jackson} B.~Jackson, B.~Stevens, and G.~Hurlbert, ``Research problems on Gray codes and universal cycles," {\it Discrete Math.} {\bf 309} (2009), 5341--5348.
\bibitem {Hurlbert} D.~Curtis, T.~Hines, G.~Hurlbert, and T.~Moyer, ``Near-Universal Cycles for Subsets Exist," {\it SIAM J. Discrete Math.} {\bf 23} (2009), 1441--1449.
\bibitem {west} D.B.~West, ``Introduction to Graph Theory," 2nd Edition, {\it Prentice Hall}, NJ, (2001).
\bibitem {bechel} A. Bechel, B. LaBounty-Lay and A. Godbole (2008), “Universal cycles of discrete functions,” Congressus Numerantium 189, 121–128.
\bibitem{hollos} S.~Hollos,``Walks on Infinite Lattices," Preprint, 2007, {\tt http://www. exstrom.com/math/lattice/inflatwalks.pdf}.
\bibitem{oxley} J.~Oxley,``What is a matroid?," {\it Cubo 5}  (2003), 179-218.
\bibitem{leitner} A.~Leitner and A.P.~Godbole,``Universal Cycles of Classes of Restricted Words," to appear in {\it Discrete Math.} (2010).
\bibitem{Hurlbert} G.~Hulbert,``On universal cycles for $k$-subsets of an $n$-set" {\it SIAM J. Discrete Math.} {\bf 7} (1994), 598-604.
\bibitem{Jackson} B.~Jackson,``Universal cycles for $k$-subsets and $k$-permutations" {\it Discrete Math.} {\bf 117} (1993), 141-150.
\bibitem{Hurlbert} G.~Hulbert, T.~Johnson, and J.~Zahl, ``On universal cycles for multisets" {\it Discrete Math.} {\bf 309} (2009), 5321--5327.
\bibitem{Narayana} T.~V.~Narayana, ``Lattice Path Combinatorics with Statistical Applications". {\it University of Toronto Press}, Toronto (1979).
\end{thebibliography}
\end{document}